\documentclass[10pt,reqno,a4paper]{amsart}

\usepackage[pdftex]{graphicx}
\usepackage{pgfplots}
\usepackage{tikz}  
\usepackage{caption}

\pgfplotsset{compat=1.18}

\usepackage{amsmath}
\usepackage{amssymb,amsthm,hyperref,caption}
\usepackage{color}
\definecolor{meinBlau}{rgb}{0.2,0.2,0.9} 
\definecolor{blau}{rgb}{0,0,0.75} 
\definecolor{rot}{rgb}{0.74,0,0} 
\hypersetup{colorlinks,linkcolor=blau,citecolor=blue,urlcolor=meinBlau}
\usepackage{times}
\usepackage{enumitem}

\allowdisplaybreaks

\newtheorem{theorem}{Theorem}
\newtheorem{lem}[theorem]{Lemma}

\newtheorem{prop}{Proposition}

\theoremstyle{definition}

\newtheorem{remark}{Remark}

\newtheorem{defi}{Definition}

\def\P{{\mathbb {P}}}
\def\E{{\mathbb {E}}}

\newcommand{\fallfak}[2]{\ensuremath{#1^{\underline{#2}}}}
\newcommand{\auffak}[2]{\ensuremath{#1^{\overline{#2}}}}

\newcommand{\N}{\ensuremath{\mathbb{N}}}

\DeclareMathOperator{\law}{\overset{\mathcal{L}}{=}}
\DeclareMathOperator{\claw}{\overset{\mathcal{L}}{\rightarrow}}

\DeclareMathOperator{\Levy}{\text{L\'evy}}

\DeclareMathOperator{\LinExp}{LinExp}

\DeclareMathOperator{\Bin}{B}
\DeclareMathOperator{\Be}{Be}

\begin{document}

\author[M.~Kuba]{Markus Kuba}
\address{Markus Kuba\\
Department Applied Mathematics and Physics\\
University of Applied Sciences - Technikum Wien\\
H\"ochst\"adtplatz 5, 1200 Wien} %
\email{kuba@technikum-wien.at}

\author[A.~Panholzer]{Alois Panholzer}
\address{Alois Panholzer\\
Institut f{\"u}r Diskrete Mathematik und Geometrie\\
Technische Universit\"at Wien\\
Wiedner Hauptstr. 8-10/104\\
1040 Wien, Austria} \email{Alois.Panholzer@tuwien.ac.at}

\title[Limit law for one-time riffle shuffle]{On Card guessing games: limit law for one-time riffle shuffle}


\begin{abstract}
We consider a card guessing game with complete feedback. A ordered deck of
$n$ cards labeled $1$ up to $n$ is riffle-shuffled exactly one time. 
Then, the goal of the game is to maximize the
number of correct guesses of the cards, where one after another a single card is drawn from the top, 
and shown to the guesser until no cards remain. 
Improving earlier results, we provide a limit law for the number of correct guesses. 
As a byproduct, we relate the number of correct guesses in this card guessing game 
to the number of correct guesses under a two-color card guessing game with complete feedback.
Using this connection to two-color card guessing, we can also show a limiting distribution result for the first occurrence of a pure luck guess.
\end{abstract}

\maketitle

\section{Introduction}
Different card guessing games have been considered in the literature in many articles~\cite{Diaconis1978,DiaconisGraham1981,HeOttolini2021,KnoPro2001,PK2023,KuPanPro2009,Leva1988,OttoliniSteiner2022,OT2023,Read1962,Zagier1990}. 
An often discussed setting is the following. A deck of a total of $M$ cards is shuffled, 
and then the guesser is provided with the total number of cards $M$, as well as with the individual numbers 
of say hearts, diamonds, clubs and spades. After each guess of the type of the next card, the person guessing the cards is shown the drawn card, which is then removed from the deck. This process is continued until no more cards are left. 
Assuming that the guesser tries to maximize the number of correct guesses, one is interested in the total number of correct guesses. 
Such card guessing games are not only of purely mathematical interest, but there are applications to the analysis of clinical trials~\cite{BlackwellHodges1957,Efron1971}, fraud detection related to 
extra-sensory perceptions~\cite{Diaconis1978}, guessing so-called Zener Cards~\cite{OttoliniSteiner2022}, 
as well as relations to tea tasting and the design of statistical experiments~\cite{Fisher1936,OT2023}.

\smallskip

The card guessing procedure can be generalized to an arbitrary number $n\ge 2$ of different types of cards. 
In the simplest setting there are two colors, red (hearts and diamonds) and black (clubs and spades), 
and their numbers are given by non-negative integers $m_1$, $m_2$, with $M=m_1+m_2$. One is then interested in the random variable $C_{m_1,m_2}$, counting the number of correct guesses. Here, not only the distribution and the expected value of the number of correct guesses is known~\cite{DiaconisGraham1981,KnoPro2001,Leva1988,Read1962,Zagier1990}, but also multivariate limit laws and additionally interesting relations to combinatorial objects such as Dyck paths
and urn models are given~\cite{DiaconisGraham1981,PK2023,KuPanPro2009}. For the general setting of $n$ different types of cards we refer the reader to~\cite{DiaconisGraham1981,HeOttolini2021,OttoliniSteiner2022,OT2023} for recent developments.

\smallskip

Different models of card guessing games involving so-called riffle shuffles are also of importance and are the main topic of this work. 
Liu~\cite{Liu2021} and also Krityakierne and Thanatipanonda~\cite{KT2023} 
studied a card guessing game carried out after a single riffle shuffle under the famous \emph{Gilbert–Shannon–Reeds} model (see Subsection~\ref{Subsection_GSR} for details): one starts with an ordered deck of $n$ cards, labeled one up to $n$, and the deck is once riffle shuffled. The number of correct guesses $X_n$, assuming that complete feedback is given, i.e., the drawn card is shown to the guessing person, and further assuming that the guesser is using the optimal strategy, is then of interest. An analysis of this procedure including an asymptotic expansion
of the expected value $\E(X_n)$ has been given in~\cite{Liu2021}. An enumerative analysis and a study of higher moments has been carried out in~\cite{KT2023}. Therein, precise asymptotics of the first few moments are provided using both enumerative and symbolic methods. 

\smallskip

In this work we provide more insight into the number of correct guesses, starting with $n$ cards labeled one up to $n$, once riffle shuffled.
We translate the enumerative analysis of~\cite{KT2023} into a distributional equation. Using a direct link between the number of correct guesses $C_{m_1,m_2}$ in the two-color card guessing game and the corresponding quantity $X_{n}$ in the once riffle shuffled model, previously unknown best to the knowledge of the authors, we obtain a limit law for the number $X_n$ of correct guesses in the once riffle shuffled case. 

\smallskip 

\subsection{Notation}
As a remark concerning notation used throughout this work, we always write $X \law Y$ to express equality in distribution of two random variables (r.v.) $X$ and $Y$, and $X_{n} \claw X$ for the weak convergence (i.e., convergence in distribution) of a sequence
of random variables $X_{n}$ to a r.v.\ $X$. Furthermore we use $\fallfak{x}{s}:=x(x-1)\dots(x-(s-1))$ for the falling factorials, and $\auffak{x}{s}:=x(x+1)\dots(x+s-1)$ for the rising factorials, $s\in\N_0$. Moreover, $f_{n} \ll g_{n}$ denotes that a sequence $f_{n}$ is asymptotically smaller than a sequence $g_{n}$, i.e., $f_{n} = o(g_{n})$, $n \to \infty$.

\section{Distributional analysis}
\subsection{Riffle shuffle model~\label{Subsection_GSR}}
A riffle shuffle is a certain card shuffling technique. In the mathematical modeling of card shuffling, the \emph{Gilbert–Shannon–Reeds} model~\cite{DiaconisGraham1981,Gilbert1955} describes a probability distribution for the outcome of such a shuffling. 
We consider a sorted deck of $n$ cards labeled consecutively from 1 up to $n$. 
The deck of cards is cut into two packets, assuming that
the probability of selecting $k$ cards in the first packet and $n-k$ in the second packet is defined as a binomial distribution with parameters $n$ and $1/2$:
\[
\frac{\binom{n}k}{2^n},\quad 0\le k\le n.
\]
Afterwards, the two packets are interleaved back into a single pile: one card at a time 
is moved from the bottom of one of the packets to the top of the shuffled deck, such that if $m_1$ cards remain in the first and $m_2$ cards remain in the second packet, then the probability of choosing a card from the first packet is $m_1/ ( m_1 + m_2 )$ and the probability of choosing a card from the second packet is $m_2 / ( m_1 + m_2 )$. See Figure~\ref{fig:RiffleMerge} for an example of a riffle shuffle of a deck of five cards.
\begin{figure}[!htb]
\includegraphics[scale=0.55]{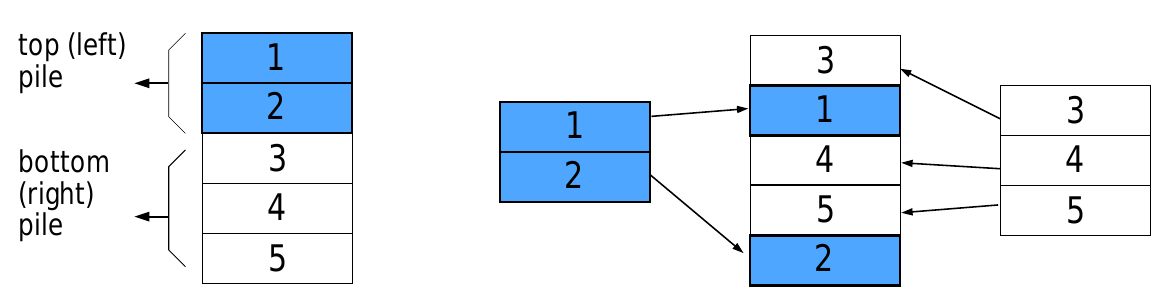}%
\caption{Example of a one-time riffle shuffle: a deck of five cards is split after 2 with probability $\binom52/2^5=5/16$ and then interleaved.}%
\label{fig:RiffleMerge}%
\end{figure}
For a one-time shuffle, the operation of interleaving described above gives rise to an ordered
deck (corresponding to the identity permutation) with multiplicity $n + 1$. 
Each other shuffled deck corresponds to a permutation which contains exactly two proper increasing subsequences
and each has multiplicity 1; in total there are $2^n- n - 1$ different such permutations.

In a more combinatorial setting, the outcome of a one-time shuffling in this model might be generated from the $2^{n}$ different $\{a,b\}$-sequences of length $n$, i.e., length-$n$ words over the alphabet $\{a,b\}$, by replacing the $a$'s in such a sequence, let us assume there are $0 \le k \le n$ many, by the increasing sequence $1, 2, \dots, k$, and the $b$'s in the sequence by the increasing sequence $k+1, k+2, \dots, n$. Thus, the $a$'s and $b$'s, respectively, correspond to the packet of cards below and above the cut, respectively. Let us denote by $\mathcal{D}_{n}$ this multiset of permutations on $[n]=\{1, 2, \dots, n\}$ generated by the family $\mathcal{W}_{n} = \{a,b\}^{n}$ of length-$n$ words. Then the $n+1$ words in $\mathcal{W}_{n}$ of the kind $a^{k} b^{n-k}$, with $0 \le k \le n$, all generate the identity permutation $\text{id}_{n}$ in $\mathcal{D}_{n}$, whereas the remaining $2^{n}-n-1$ words in $\mathcal{W}_{n}$ generate pairwise different permutations in $\mathcal{D}_{n}$.

\subsection{First drawn card and the optimal strategy}
The optimal strategy for maximizing the number $X_n$ of correctly guessed cards, starting with a deck of $n$ cards, after a one-time riffle shuffle stems from the following Proposition.
\begin{prop}[Guessing the first card~\cite{KT2023,Liu2021}]
\label{Prop:firstCard}
Assume that a deck of $n$ cards has been riffle shuffled once. 
The probability $p_n(m)$ that the first card being $m$, $1\le m \le n$, 
is given by
\begin{equation}
p_n(m) = 
\begin{cases}
\displaystyle{\frac12 + \frac1{2^n}},\quad \text{for }m=1,\\[0.2cm]
\displaystyle{\frac{\binom{n-1}{m-1}}{2^n}},\quad \text{for } 2\le m\le n.
\end{cases}
\end{equation}
\end{prop}
For the sake of completeness we include a short proof.
\begin{proof}
First, we condition on the cut leading to two decks containing $\{1,\dots, k\}$ and $\{k+1,\dots, n\}$, $0\le k\le n$, which happens with probability $\frac{\binom{n}{k}}{2^{n}}$. 
Each resulting deck has the same probability $1/\binom{n}k$. 
Then, we observe that the probability of a top one is in the case $k>0$ given by
\[
\frac{\binom{n-1}{k-1}}{\binom{n}{k}}=\frac{k}{n}, 
\]
as there are $\binom{n-1}{k-1}$ different ways of choosing the positions of the other cards. Of course, for $k=0$ the top card is always one.
Thus, we obtain
\[
p_{n}(1)=\frac1{2^n} + \sum_{k=1}^{n}\frac{\binom{n}k}{2^n} \cdot \frac{k}{n}
=\frac1{2^n}+\frac12.
\]
Similarly, for $m >1$ we observe that only a cut at $m-1$ may lead to a top card labeled $m$, thus in this situation the
subsequences to be interleaved have to be $1,\dots,m-1$ and $m,\dots, n$. If $m$ is the top card, there are $\binom{n-1}{n-m}$ different ways of choosing the positions of the other cards, which yields
\[
p_{n}(m)=\frac{\binom{n}{m-1}}{2^n}\cdot \frac{\binom{n-1}{n-m}}{\binom{n}{m-1}} =\frac{\binom{n-1}{m-1}}{2^n}.
\]
\end{proof}

Now we turn to the optimal strategy. The guesser should guess 1 on the first
card, as his chance of success is more than $50\%$ by Proposition~\ref{Prop:firstCard}

\smallskip

If the first guess is incorrect, say the shown card has label $m\ge 2$, this implies that the 
cut was made exactly at $m-1$. The person is left with two increasing subsequences $1,2,\dots,m-1$
and $m+1,\dots,n$. The remaining numbers are then guessed according to the proportions of the lengths of the remaining subsequences until no cards are left. 

\smallskip

If the first guess was correct, then the person continues with guessing the number two, etc., i.e., as long as all previous such predictions turned out to be correct, the guesser makes a guess of the number $j$ for the $j$-th card. This is justified, since by considerations as before one can show easily that the probability that the $j$-th card has the number $j$ conditioned on the event that the first $j-1$ cards are the sequence of numbers $1, 2, \dots, j-1$ is for $1 \le j \le n$ given by
\begin{equation*}
  \frac{2^{n-j}+j}{2^{n-j+1}+j-1} = \frac{1}{2} + \frac{(1+j)2^{-(n-j+2)}}{1+(j-1)2^{-(n-j+1)}},
\end{equation*}
and thus exceeds $50\%$.
If such a prediction turns out to be wrong, i.e., gives a number $m>j$ for the $j$-th card, then again one can determine the two involved remaining subsequences $j, j+1, \dots, m-1$ and $m+1, \dots, n$,
and all the numbers of the remaining cards are again guessed according to the proportions of the lengths of the remaining subsequences
until no cards are left.

\subsection{Enumeration and distributional decomposition}
Our starting point is the recurrence relation for the generating function
\begin{equation*}
  D_{n}(q) := \sum_{\sigma \in \mathcal{D}_{n}} q^{\text{$\#$ correct guesses for deck $\sigma$}} = 2^{n} \cdot \E(q^{X_{n}}) = 2^{n} \sum_{\ell=0}^{n} \P(X_{n}=\ell) \, q^{\ell},
\end{equation*}
counting the number of correct guesses using the optimal strategy when starting with a once-shuffled deck of $n$ different cards, 
which has been stated in \cite{KT2023} and basically stems from Proposition~\ref{Prop:firstCard}.
\begin{lem}[Recurrence relation for $D_n(q)$~\cite{KT2023}]\label{lem:Recurrence_Dn}
The generating function $D_{n}(q)$ satisfies the following recurrence:
\begin{equation}
D_n(q)=q D_{n-1}(q) + q^{n} + \sum_{j=0}^{n-2}F_{n-1-j,j}(q), \quad n \ge 1, \qquad D_{0}(q)=1,
\label{eq:Start}
\end{equation}
where the auxiliary function $F_{m_{1},m_{2}}(q)$ is for $m_{1} \ge m_{2} \ge 0$ defined recursively by 
\[
F_{m_{1},m_{2}}(q)=q F_{m_{1}-1,m_{2}}(q)+F_{m_{1},m_{2}-1}(q),
\]
with initial values $F_{m_{1},0}(q)=q^{m_{1}}$, and for $m_{2} > m_{1} \ge 0$ by the symmetry relation
\begin{equation*}
  F_{m_{1},m_{2}}(q)=F_{m_{2},m_{1}}(q).
\end{equation*} 
\end{lem}
\begin{proof}
  To keep the work self-contained we give a proof of this recurrence, where we use the before-mentioned combinatorial description of once-shuffled decks of $n$ cards $\sigma = (\sigma_{1}, \dots, \sigma_{n}) \in \mathcal{D}_{n}$ by means of length-$n$ words $w=w_{1} \dots w_{n} \in \mathcal{W}_{n}$. We count the number of correct guesses, where we distinguish according to the first letter $w_{1}$. If $w_{1}=a$ then the first drawn card is $1$, $\sigma_{1}=1$, and this card will be predicted correctly by the guesser. The guesser keeps his strategy of guessing for the deck of remaining cards, which is order-isomorphic to a deck of $n-1$ cards generated by the length-$(n-1)$ word $w'=w_{2} \dots w_{n}$; to be more precise, if $\sigma = (1, \sigma_{2}, \dots, \sigma_{n}) \in \mathcal{D}_{n}$ and $\sigma' = (\sigma_{1}', \dots, \sigma_{n-1}') \in \mathcal{D}_{n-1}$ are the labels of the cards in the deck generated by the words $w = aw' \in \mathcal{W}_{n}$ and $w' \in \mathcal{W}_{n-1}$, respectively, then it simply holds $\sigma_{i} = \sigma_{i-1}'+1$, $2 \le i \le n$. Since $w'$ is a random word of length $n-1$ if started with a random word $w$ of length $n$, this yields the summand $q D_{n-1}(q)$ in equation~\eqref{eq:Start}.

If $w_{1}=b$ then we first consider the particular case that $w=b^{n}$, i.e., that the cut of the deck has been at $0$. Since in this case the deck of cards corresponds to the identity permutation $\sigma = \text{id}_{n}$, the guesser will predict all cards correctly using the optimal strategy, which leads to the summand $q^{n}$ in \eqref{eq:Start}. Apart from this particular case, $w_{1}=b$ corresponds to a deck of cards where the first card is $m \ge 2$ and thus will cause a wrong prediction by the guesser; however, due to complete feedback, now the guesser knows that the cut is at $m-1$, or in alternative terms, he knows that the remaining deck is generated from a word $w'=w_{2} \dots w_{n}$ that has $j:=n-m$ $b$'s and $n-1-j=m-1$ $a$'s, with $0 \le j \le n-2$. From this point on the guesser changes the strategy, which again could be formulated in alternative terms by saying that the guesser makes a guess for the next letter in the word, in a way that the guess is $a$ if the number of $a$'s exceeds the number of $b$'s in the remaining subword, that the guess is $b$ in the opposite case, and (in order to keep the outcome deterministic) that the guess is $a$ if there is a draw between the number of $a$'s and $b$'s. More generally, let us assume that the word consists of $m_{1} \ge 0$ $a$'s and $m_{2} \ge 0$ $b$'s and each of these $\binom{m_{1}+m_{2}}{m_{1}}$ words occur with equal probability, then let us define the r.v.\ $\hat{C}_{m_{1},m_{2}}$ counting the number of correct guesses by the before-mentioned strategy as well as the generating function $F_{m_{1},m_{2}}(q) = \binom{m_{1}+m_{2}}{m_{1}} \E(q^{\hat{C}_{m_{1},m_{2}}})$. It can be seen immediately that $\hat{C}_{m_{1},m_{2}}$ and so $F_{m_{1},m_{2}}(q)$ is symmetric in $m_{1}$ and $m_{2}$, and that $F_{m_{1},m_{2}}(q)$ satisfies the recurrence stated in \eqref{eq:Start}. Moreover, these considerations yield the third summand $\sum_{j=0}^{n-2} F_{n-1-j,j}(q)$ in equation~\eqref{eq:Start}.
\end{proof}

When considering the two-color card guessing game (with complete feedback) starting with $m_{1}$ cards of type (color) $a$ and $m_{2}$ cards of type (color) $b$ it apparently corresponds to the guessing game for the letters of a word over the alphabet $\{a,b\}$ consisting of $m_{1}$ $a$'s and $m_{2}$ $b$'s as described in the proof of Lemma~\ref{lem:Recurrence_Dn}. Thus, the r.v.\ $C_{m_{1},m_{2}}$ counting the number of correct guesses when the guesser uses the optimal strategy for maximizing correct guesses, i.e., guessing the color corresponding to the larger number of cards present~\cite{DiaconisGraham1981,KnoPro2001,PK2023,KuPanPro2009}, and the r.v.\ $\hat{C}_{m_{1},m_{2}}$ are equally distributed, $C_{m_{1},m_{2}} \law \hat{C}_{m_{1},m_{2}}$. Consequently, the auxiliary function $F_{m_{1},m_{2}}(q)$ stated in Lemma~\ref{lem:Recurrence_Dn} is the generating function of $C_{m_{1},m_{2}}$:
\begin{equation}\label{eqn:Cm1m2_Fm1m2q_relation}
  F_{m_{1},m_{2}}(q) = \binom{m_{1}+m_{2}}{m_{1}} \cdot \E(q^{C_{m_{1},m_{2}}}).
\end{equation}

\begin{remark}
In most works considering $C_{m_1,m_2}$ it is assumed without loss of generality that $m_1\ge m_2\ge 0$. 
However, we note that by definition of the two-color card guessing game the order of the parameters is not of relevance under
the optimal strategy: $C_{m_1,m_2}=C_{m_2,m_1}$.
\end{remark}

\begin{remark}\label{rem:Urn_LatticPath}
As has been pointed out in \cite{PK2023}, the two-color guessing procedure for the cards of a deck with $m_{1}$ cards of type $a$ (say color red) and $m_{2}$ cards of type $b$ (say color black) can be formulated also by means of the so-called sampling without replacement urn model starting with $m_{1}$ and $m_{2}$ balls of color red and black, respectively, where in each draw a ball is picked at random, the color inspected and then removed, until no more balls are left. Then the urn histories can be described via weighted lattice paths from $(m_1,m_2)$ to the origin with step sets ``left'' $(-1,0)$ and ``down'' $(0,-1)$: at position $(k_{1},k_{2})$, a left-step and a down-step have weights $\frac{k_{1}}{k_{1}+k_{2}}$ and $\frac{k_{2}}{k_{1}+k_{2}}$, respectively, and reflect the draw of a red ball or a black ball, resp., occurring with the corresponding probabilities. Several quantities of interest for card guessing games can be formulated also via parameters of the sample paths of this urn, such as the first hitting of the diagonal or the first hitting of one of the coordinate axis, which is used in a subsequent section.
\end{remark}

Concerning a distributional analysis of $X_{n}$, an important intermediate result is the following distributional equation, 
which we obtain by translating the recurrence relation~\eqref{eq:Start} into a recursion for probability generating functions.
\begin{theorem}[One-time riffle and two-color card guessing]
\label{the:distDecomp}
The random variable $X=X_n$ of correctly guessed cards, starting with a deck of $n$ cards, after a one-time riffle 
satisfies the following decomposition:
\begin{equation}\label{eqn:Xn_DistEqn}
X_n \law I_1\big(X^{\ast}_{n-1}+1\big)
+(1-I_1)\Big(I_2\cdot n +
(1-I_2)\cdot C_{n-1-J_n,J_n}\Big),
\end{equation}
where $I_1\law \Be(0.5)$, $I_2=\law \Be(0.5^{n-1})$, and $C_{m_1,m_2}$ denotes the number of correct guesses in a two-color card guessing game. Additionally, $X^{\ast}_{n-1}$ is an independent copy of $X$ defined on $n-1$ cards. 
Moreover, $J_n\law \Bin^{*}(n-1,p)$ denotes a truncated binomial distribution: 
\[
\P(J_n=j)=\binom{n-1}{j}/(2^{n-1}-1),\quad  0\le j \le n-2.
\]
All random variables $I_1$, $I_2$, $J_n$, as well as $C_{m_1,m_2}$ are mutually independent.
\end{theorem}

\begin{proof}
By definition, the probability generating function of $X_n$ is given as follows:
\[
\E(q^{X_n})=\frac{D_n(q)}{2^n}.
\]
Thus, we get from~\eqref{eq:Start} the equation
\begin{equation}
\E(q^{X_n})=\frac{1}2 \cdot \E(q^{X_{n-1}+1}) +\frac{1}{2^n}\cdot q^{n} + \frac{1}{2^n}\sum_{j=0}^{n-1}F_{n-1-j,j}(q).
\label{eq:Start2}
\end{equation}
As pointed out above, the probability generating function of $C_{m_{1},m_{2}}$ is given via
\begin{equation*}
\E(q^{C_{m_1,m_2}}) = \frac{F_{m_1,m_2}(q)}{\binom{m_1+m_2}{m_1}}.
\end{equation*}
Thus, the last summand in \eqref{eq:Start2} yields the following representation
\begin{align*}
&\frac{1}{2^n}\sum_{i=0}^{n-2}F_{n-1-j,j}(q)
=\frac{1}{2^n}\sum_{i=0}^{n-2}\E(q^{C_{n-1-j,j}})\cdot \binom{n-1}{j} \\
&\quad=\frac{2^{n-1}-1}{2^n}\sum_{i=0}^{n-2}\E(q^{C_{n-1-j,j}})\cdot \frac{\binom{n-1}{j}}{2^{n-1}-1}\\
&\quad=\frac12\Big(1 - \frac1{2^{n-1}}\Big)\sum_{i=0}^{n-2}\E(q^{C_{n-1-j,j}})\P\{J_n=j\}
=\frac12\Big(1 - \frac1{2^{n-1}}\Big)\E(q^{C_{n-1-J_n,J_n}}).
\end{align*}
Translating these expressions for the probability generating functions involved into a distributional equation
leads to the stated result. Note that the fact that $X_{n-1}^{\ast}$ indeed has the same distribution as $X$ defined on a deck of $n-1$ cards follows from equation~\eqref{eq:Start}.

%
\end{proof}

The distributional decomposition together with the properties of the binomial distribution and 
the limit laws of two-color card guessing game allow to obtain a limit law for $X_n$. 
By the classical de~Moivre–Laplace theorem, we can approximate the binomial distribution 
$J_n$ with mean $\frac{n}2$ and standard deviation $\sqrt{n}/2$ by a normal random variable. 
This suggests that we need to study $C_{n-1-j,j}$ for $j=\frac{n}2 + x\sqrt{n}$, as $n$ tends to infinity. 
We recall the limit law for the two-color card guessing game in the required range (see~\cite{PK2023,KuPanPro2009} for a complete discussion of all different limit laws of $C_{m_{1},m_{2}}$ depending on the growth behaviour of $m_{1}$, $m_{2}$; additionally, we also refer to~\cite{DiaconisGraham1981,Zagier1990} for the case $m_1=m_2$).
\begin{theorem}[Limit law for two-color card guessing~\cite{PK2023,KuPanPro2009}]
\label{the:LinExp}
Assume that the numbers $m_1$, $m_2$ satisfy $m_1-m_2\sim \rho \cdot \sqrt{m_1}$, as $m_1\to\infty$, with $\rho>0$. Then, the number of correct guesses $C_{m_1,m_2}$ is asymptotically linear exponentially distributed,
\[
\frac{C_{m_1,m_2}-m_1}{\sqrt{m_1}}\claw \LinExp(\rho,2),
\]
or equivalently by explicitly stating the cumulative distribution function of $\text{LinExp}(\rho,2)$:
\[
\P\{C_{m_1,m_2}\le m_1+\sqrt{m_1}z\}\to 1-e^{-z(\rho+z)}, \quad \text{for $z \ge 0$}.
\]
\end{theorem}

In order to derive a limit law for $X_n$ we require first a limit law for $C_{n-1-J_n,J_n}$ as occurring in Theorem~\ref{the:distDecomp}.
\begin{lem}
\label{lem:twoColorBinLimit}
The random variable $C_{n-1-J_n,J_n}$, with $J_{n}$ as defined in Theorem~\ref{the:distDecomp}, satisfies the following limit law:
\[
\frac{C_{n-1-J_n,J_n}-\frac{n}2}{\sqrt{n}}\to G,
\]
where $G$ denotes a generalized gamma distributed random variable with 
probability density function
\[
f(x)=\sqrt{\frac{2}{\pi}}\cdot 8x^2e^{-2x^2},\quad x\ge 0.
\]
\end{lem}

\begin{figure}[!htb]
\begin{tikzpicture}
\begin{axis}[
 domain=0:4,samples=100,
  axis lines*=left,
  height=6cm, width=11cm,
  xtick={0,0.5, 1,1.5, 2,2.5, 3,3.5, 4}, 
	 ytick={0,0.1, 0.2, 0.3, 0.4, 0.5, 0.6, 0.7,0.8,0.9,1,1.1,1.2}, 
	ymin=0, ymax=1.2,
	xmin=0, xmax=4,
   axis on top,  clip=false, no markers,
	 xlabel={$x$},
   ylabel={$f(x)$},
  ]
 \addplot [very thick, blue] {sqrt(2/pi)*8*x^2*exp(-2*x^2)};
\end{axis}
\end{tikzpicture}
\caption{Plot of the density function $f(x)$ of the generalized Gamma distribution occurring in Theorem~\ref{the:Yn_LimitLaw} and Lemma~\ref{lem:twoColorBinLimit}.}
\label{fig:GammaDensity}%
\end{figure}
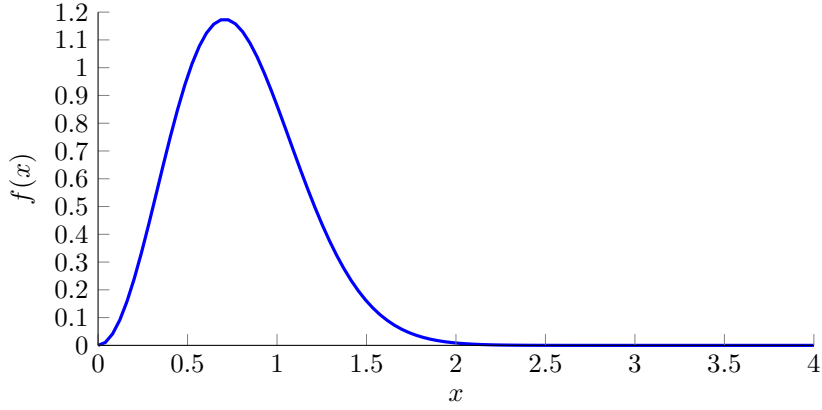
\begin{remark}
This special instance of a generalized Gamma distribution is also known as a Maxwell-Boltzmann distribution with parameter $a=1/2$,
which is of important for describing particle speeds in idealized gases.

\smallskip 

The first three raw integer moments of $G$ are 
\[
\E(G)=\mu_G=\sqrt{\frac{2}\pi}\approx 0.7979,\quad
\E(G^2)=\frac34,\quad
\E(G^3)=\sqrt{\frac{2}\pi}.
\]
Consequently, the standard deviation $\sigma_G$ and the skewness $\gamma_G$ are given by
\[
\sigma_G = \sqrt{\E(G^{2})-\mu_{G}^{2}} \approx 0.3367,\quad \gamma_G=\frac{\E(G^{3})-3\mu_{G} \E(G^{2}) + 2 \mu_{G}^{3}}{\sigma_{G}^{3}} \approx 0.4857,
\]
leading to a right-skewed distribution, in agreement with the numerical observations of the limit law of $X_{n}$ (which turns out to be $G$ as well) in~\cite{KT2023}. See Figure~\ref{fig:GammaDensity} for a plot of the density function of $G$.
\end{remark}

\begin{proof}
We consider the distribution function 
\[
F_n(x)=\P\big\{C_{n-1-J_n,J_n}\le \frac{n}2 + x\sqrt{n}\big\}
\]
for fixed positive real $x$. Conditioning on the truncated binomial distribution gives
\[
F_n(x) =\sum_{j=0}^{n-2}\P\big\{C_{n-1-j,j}\le \frac{n}2 + x\sqrt{n}\big\}\P\{J_n=j\}.
\]
We can exploit the symmetry of the binomial distribution, as well as $C_{m_1,m_2}$, 
to get
\[
F_n(x)\sim 2\cdot \sum_{j=\lfloor n/2\rfloor}^{n-2}\P\big\{C_{j,n-1-j}\le \frac{n}2 + x\sqrt{n}\big\} \cdot \P\{J_n=j\}.
\]
By the de~Moivre-Laplace limit theorem for the binomial distribution we get for large $n$
\[
F_n(x)\sim 2\int_{n/2}^{n-2}\P\big\{C_{j,n-1-j}\le \frac{n}2 + x\sqrt{n}\big\} \cdot \frac{e^{-\frac{(j-\mu_{n})^2}{2\sigma_{n}^2}}}{\sigma_{n}\sqrt{2\pi}}dj,
\]
where $\mu_{n}=n/2$ and $\sigma_{n}=\sqrt{n}/2$. Substituting $j=\mu_{n}+t\sigma_{n}$, we obtain further
\[
F_n(x)\sim 2\int_{0}^{\infty}\P\big\{C_{n/2 + t\sqrt{n}/2,n-1-n/2-t\sqrt{n}/2}\le \frac{n}2 + x\sqrt{n}\big\} \cdot \frac{e^{-\frac{t^2}{2}}}{\sqrt{2\pi}} dt.
\]
Next, we asymptotically evaluate the integrand by using the limit law from Theorem~\ref{the:LinExp} for the two-color card guessing game with 
\[
m_1=n/2 + t\sqrt{n}/2,\quad m_2=n-1-n/2-t\sqrt{n}/2.
\]
Since $C_{m_{1},m_{2}} \ge \max\{m_{1},m_{2}\}$ (see, e.g., \cite{PK2023}), we deduce that for $t>2x$ it holds
\[
\P\big\{C_{n/2 + t\sqrt{n}/2,n-1-n/2-t\sqrt{n}/2}\le \frac{n}2 + x\sqrt{n}\big\}\sim 0.
\]
Furthermore, in the range $0\le t\le 2x$ we obtain from Theorem~\ref{the:LinExp}, by setting $\rho = \sqrt{2} t$ and $z=\sqrt{2}(x-t/2)$,
\begin{multline*}
\P\big\{C_{n/2 + t\sqrt{n}/2,n-1-n/2-t\sqrt{n}/2}\le \frac{n}2 + x\sqrt{n}\big\} \\
\to 1-\exp\Big(-\sqrt{2}\big(x-\frac{t}2\big)\big(\sqrt{2} t + \sqrt{2}(x-\frac{t}{2})\big)\Big) = 1-\exp\Big(-2x^2+\frac{t^2}2\Big).
\end{multline*}
This implies that
\begin{align*}
F_n(x)&\sim \frac{2}{\sqrt{2\pi}}\cdot \int_0^{2x}e^{-t^2/2}\Big(1-\exp\Big(-2x^2+\frac{t^2}2\Big)\Big)dt\\
& =\frac{2}{\sqrt{2\pi}}\cdot \int_0^{2x}\Big(e^{-t^2/2}-e^{-2x^2}\Big)dt 
= \frac{2}{\sqrt{2\pi}}\cdot \Big( \int_0^{2x}e^{-t^2/2}dt -\frac{2x}{\sqrt{2\pi}}e^{-2x^2}\Big).
\end{align*}
Differentiating the last expression with respect to $x$ leads to the desired density function of the limiting r.v.\ $G$,
\[
f(x)=\frac{2}{\sqrt{2\pi}}\Big(e^{-2x^2}\cdot 2 - 2e^{-2x^2} + 8x^2\cdot e^{-2x^2}\Big)
=\sqrt{\frac{2}{\pi}}\cdot 8x^2e^{-2x^2}.
\]
\end{proof}

Next we state the main result of this work, a limit law for the number of correct guesses $X_n$. 
The limit law is the same as in Lemma~\ref{lem:twoColorBinLimit}, involving the generalized Gamma distribution.

\begin{theorem}\label{the:Yn_LimitLaw}
The normalized random variable $Y_n=(X_n-\frac{n}2)/\sqrt{n}$ converges in distribution
to a generalized gamma distributed random variable $G$, $Y_n\claw G$, with density 
$f(x)=\sqrt{\frac{2}{\pi}}\cdot 8x^2e^{-2x^2}$, $x\ge 0$.
\end{theorem}
\begin{remark}[A fixed-point equation]
Once we know that the limit law exists, one can informally derive the limit law from the distributional equation~\eqref{eqn:Xn_DistEqn} by omitting asymptotically negligible terms:
\[
Y_n \sim I_1 \cdot Y_{n-1} + (1-I_1)\frac{C_{n-1-J_n,J_n}-\frac{n}2}{\sqrt n},
\]
where $I_1=\Be(0.5)$. Thus, for large $n$ we anticipate a sort of fixed-point equation for the limit law $Y$ of $Y_n$:
\[
Y \sim I_1 \cdot Y + (1-I_1)\cdot G, 
\]
with $G$ the generalized Gamma limit law. Similarly, we may anticipate that all integer moments of $Y$ are simply the moments of $G$:
\[
\E(Y^{r}) = \frac12 \E(Y^{r}) + \frac12 \E(G^{r}), \quad \text{and further} \quad \E(Y^{r})=\E(G^{r}), \quad r \ge 0.
\]
\end{remark}
\begin{proof}
According to Theorem~\ref{the:distDecomp} we get
\begin{align*}
& \P\big\{X_n\le \frac{n}2+x\sqrt{n}\big\}\\
& \quad = \frac{1}{2} \P\big\{X_{n-1}+1\le \frac{n}2+x\sqrt{n}\big\}
 + \Big(\frac{1}{2}-\frac{1}{2^{n}}\Big)\P\big\{C_{n-1-J_n,J_n}\le \frac{n}2 + x\sqrt{n}\big\}.
\end{align*}
Moreover, by iterating this recursive representation we observe that, for $n\to\infty$,
\[
\P\big\{X_n\le \frac{n}2+x\sqrt{n}\big\}
\sim \sum_{\ell \ge 1}\frac{1}{2^\ell} \cdot
\P\big\{C_{n-\ell-J_{n-\ell},J_{n-\ell}}\le \frac{n}2 + x\sqrt{n}\big\}.
\]
As $n$ tends to infinity, Lemma~\ref{lem:twoColorBinLimit} ensures that all the distribution functions occurring converge to the same limit, from which the stated result follows.
\end{proof}

\subsection{Moment convergence}
Krityakierne and Thanatipanonda~\cite{KT2023} provided extremely precise results for the first few (factorial) moments of $X_n$, 
as well as for the centered moments $\E((X_n-\mu)^{r})$, for $r=1,2,3$. We state a simplified version of their result:
\begin{gather}
  \mu=\E(X_n)=\frac{n}{2}+\sqrt{\frac{2n}{\pi}} - \frac12 +\mathcal{O}(n^{-1/2}), \quad \E\big((X_n-\mu)^2\big)=\Big(\frac34-\frac2\pi\Big)n+\mathcal{O}(1),\notag\\
	\E\big((X_n-\mu)^3\big)=\sqrt{\frac2\pi}\Big(\frac4\pi-\frac54\Big)n^{3/2} + \mathcal{O}(n^{1/2}).\label{eqn:mom1}
\end{gather}

First we use above expansions of $\E((X_n-\mu)^{r})$ to determine the asymptotics of the first moments of $Y_{n}=(X_n-\frac{n}2)/\sqrt{n}$ in a straightforward way. One observes that the limits of $\E\big(Y_{n}^{r}\big)$, $r=1,2,3$, are in agreement with the limit law $G$ stated in Theorem~\ref{the:Yn_LimitLaw}.  
\begin{prop}
Let $Y_n=(X_n-\frac{n}2)/\sqrt{n}$. The moments $\E(Y_{n}^{r})$ converge for $r=1,2,3$ to the moments of the limit law $G$:
\[
\E\big(Y_n\big)\to\sqrt{\frac{2}\pi}=\E(G),\quad
\E\big(Y_n^2\big) \to \frac34=\E(G^2),\quad
\E\big(Y_n^{3}\big)\to \sqrt{\frac{2}\pi}=\E(G^3).
\]
\end{prop}
\begin{proof}
The result for the expected value $\E(Y_n)$ follows directly from~\eqref{eqn:mom1}.
In the following let $\mu = \E(X_{n})= \frac{n}{2} + \delta_{n}$. Due to \eqref{eqn:mom1} it holds
\begin{equation}\label{eqn:d}
  \delta_{n} = \sqrt{\frac{2n}{\pi}} - \frac12 +\mathcal{O}(n^{-1/2}).
\end{equation}
Consequently, the second centered moment can be rewritten as follows:
\[
\E\big((X_n-\mu)^2\big)=
\E\big((X_n-\frac{n}2 -\delta_{n})^2\big)
=\E\big((X_n-\frac{n}2)^{2}\big)-2 \delta_{n}\E\big(X_n-\frac{n}2\big)+\delta_{n}^{2},
\]
which gives, by using expansions \eqref{eqn:mom1} and \eqref{eqn:d},
\begin{align*}
\E\big(Y_n^2\big)&=
\frac1n\E\big((X_n-\frac{n}2)^{2}\big)
=\frac1n\Big[\E\big((X_n-\mu)^2\big)+2 \delta_{n}\E\big(X_n-\frac{n}{2}\big)-\delta_{n}^{2}\Big]\\
&=\frac1n\Big[\E\big((X_n-\mu)^2\big)+ \delta_{n}^{2}\Big]\sim \frac34.
\end{align*}
In a similar way, by rewriting the third centered moment and using \eqref{eqn:mom1} and \eqref{eqn:d}, one obtains the stated result for $\E(Y_{n}^{3})$.
\end{proof}

Actually, in the following we are going to show that indeed all integer moments of $Y_{n}$ converge to the corresponding moments of the limit law $G$. Let us first state them.
\begin{prop}
  The integer moments of the generalized gamma distributed random variable $G$ with probability density function as defined in Lemma~\ref{lem:twoColorBinLimit} are given as follows:
\begin{equation*}
  \E\big(G^{r}\big) = \frac{\Gamma\big(\frac{r+3}{2}\big)}{2^{\frac{r}{2}-1} \sqrt{\pi}}, \quad r \ge 0.
\end{equation*}
\end{prop}
\begin{proof}
  A straightforward evaluation of the defining integral of the $r$-th moment of $G$ by means of the $\Gamma$-function after substituting $t=2x^{2}$ yields the stated result:
\begin{align*}
  \E\big(G^{r}\big) & = \int_{0}^{\infty} f(x) x^{r} dx = 8 \sqrt{\frac{2}{\pi}}\cdot \int_{0}^{\infty} x^{r+2} e^{-2x^2} dx
	= \frac{1}{2^{\frac{r}{2}-1} \sqrt{\pi}} \cdot \int_{0}^{\infty} t^{\frac{r+1}{2}} e^{-t} dt \\
	& = \frac{\Gamma\big(\frac{r+3}{2}\big)}{2^{\frac{r}{2}-1} \sqrt{\pi}}.
\end{align*}
\end{proof}

\begin{theorem}\label{the:Yn_Moments}
Let $Y_n=(X_n-\frac{n}2)/\sqrt{n}$. The $r$-th integer moments $\E(Y_{n}^{r})$ converge, for arbitrary but fixed $r$ and $n \to \infty$, to the moments of the limit law $G$:
\begin{equation*}
\E\big(Y_{n}^{r}\big) \to \E\big(G^{r}\big) = \frac{\Gamma\big(\frac{r+3}{2}\big)}{2^{\frac{r}{2}-1} \sqrt{\pi}}, \quad r \ge 0.
\end{equation*}
\end{theorem}
\begin{remark}
  Since the generalized gamma distributed r.v.\ $G$ is uniquely characterized by its moments (which easily follows, e.g., from simple growth bounds), we note that an application of the moment's convergence theorem of Fr\'echet and Shohat (see, e.g., \cite{Loe1977}) immediately shows convergence in distribution of $Y_{n}$ to $G$, thus gives an alternative proof of Theorem~\ref{the:Yn_LimitLaw}.
\end{remark}

 To show Theorem~\ref{the:Yn_Moments} we will again start with the recursive description of $D_{n}(q)$ given in Lemma~\ref{lem:Recurrence_Dn}, but in order to deal with this recurrence we use an alternative approach based on generating functions and basic techniques from analytic combinatorics~\cite{FlaSed}. Furthermore, we use explicit formul{\ae} for a suitable bivariate generating function of $\E\big(q^{C_{m_{1},m_{2}}}\big)$ and the so-called diagonal as have been derived in \cite{KnoPro2001,PK2023}. They can be stated in the following form.
\begin{prop}[\cite{KnoPro2001,PK2023}]\label{pro:Fxyq_formulae}
The g.f.\
$\tilde{F}(x,y,q) = \sum\limits_{m_{1} \ge m_{2} \ge 0} \binom{m_{1}+m_{2}}{m_{1}} \E\big(q^{C_{m_{1},m_{2}}}\big) x^{m_{1}} y^{m_{2}}$
and $\tilde{F}_{0}(x,y,q) = \sum\limits_{m \ge 0} \binom{2m}{m} \E\big(q^{C_{m,m}}\big) x^{m} y^{m}$
are given as follows:
\begin{align*}
  \tilde{F}(x,y,q) & = \frac{1-y}{1-qx-y} + \frac{qxy(q-(1+q)y)}{(1-qx-y)(1-B(qxy))(1-(1+q)B(qxy))}, \\
	\tilde{F}_{0}(x,y,q) & = \frac{1}{1-(1+q)B(qxy)}, \\
\end{align*}
where $B(t) = \frac{1-\sqrt{1-4t}}{2} = \sum_{n \ge 1} \frac{1}{n} \binom{2n-2}{n-1} t^{n}$ denotes the g.f.\ of the shifted Catalan-numbers.
\end{prop}

With these results we obtain a generating functions solution of recurrence~\eqref{eq:Start} for $D_{n}(q)$.
\begin{lem}\label{lem:Dzq_formula}
  The bivariate generating function
	\begin{equation*}
	  D(z,q) = \sum_{n \ge 0} D_{n}(q) z^{n} = \sum_{n \ge 0} 2^{n} \E\big(q^{X_{n}}\big) z^{n}
	\end{equation*}
	is given by the following explicit formula, with $B(t) = \frac{1-\sqrt{1-4t}}{2}$:
	\begin{multline*}
	  D(z,q) = \frac{1-z}{(1-qz)^{2}} + \frac{z}{1-qz} \left[\frac{2(1-z)}{1-(1+q)z}\right.\\
		\left.\mbox{} + \frac{2qz^{2}(q-(1+q)z)}{(1-(1+q)z)(1-(1+q)B(qz^{2}))(1-B(qz^{2}))} - \frac{1}{1-(1+q)B(qz^{2})}\right].
	\end{multline*}
\end{lem}
\begin{proof}
  Introducing the auxiliary g.f.\ $F(x,y,q) = \sum_{m_{1} \ge 0} \sum_{m_{2} \ge 0} F_{m_{1},m_{2}}(q) x^{m_{1}} y^{m_{2}}$, we obtain from recurrence~\eqref{eq:Start} after multiplying with $z^{n}$ and summing over integers $n \ge 0$ the relation
	\begin{equation*}
	  D(z,q) = qz D(z,q) + \frac{1-z}{1-qz} + z F(z,z,q),
	\end{equation*}
	and further
	\begin{equation}\label{eqn:Dzq_relation}
	  D(z,q) = \frac{1-z}{(1-qz)^{2}} + \frac{z F(z,z,q)}{1-qz}.
	\end{equation}
	Using the relation
	\begin{equation*}
	  F(x,y,q) = \tilde{F}(x,y,q) + \tilde{F}(y,x,q) - \tilde{F}_{0}(x,y,q),
	\end{equation*}
	which is immediate from the definitions given in \eqref{eqn:Cm1m2_Fm1m2q_relation} and Proposition~\ref{pro:Fxyq_formulae}, and the explicit formul{\ae} given in Proposition~\ref{pro:Fxyq_formulae}, the stated result follows from \eqref{eqn:Dzq_relation}.
\end{proof}
We are interested in the asymptotic behaviour of the moments of the shifted r.v.\ $\hat{X}_{n} := X_{n} - n/2 = \sqrt{n} \, Y_{n}$. The corresponding g.f.\ $\hat{D}(z,q)$ is closely related to $D(z,q)$ as defined in Lemma~\ref{lem:Dzq_formula}, since we get
\begin{equation}\label{eqn:HatDzq_definition}
  \hat{D}(z,q) = \sum_{n \ge 0} 2^{n} \E\big(q^{\hat{X}_{n}}\big) z^{n} = \sum_{n \ge 0} 2^{n} \E\big(q^{X_{n}}\big) q^{-\frac{n}{2}} z^{n} = D\Big(\frac{z}{\sqrt{q}}, q\Big).
\end{equation}
Actually, we will set $q=1+u$ and use that the coefficients of the probability generating function in a series expansion around $u=0$ yield the factorial moments of $\hat{X}_{n}$:
\begin{equation*}
  \E\big((1+u)^{\hat{X}_{n}}\big) = \sum_{r \ge 0} u^{r} \, \E\Big({\textstyle{\binom{\hat{X}_{n}}{r}}}\Big) = \sum_{r \ge 0} \E\big(\hat{X}_{n}^{\underline{r}}\big) \frac{u^{r}}{r!}.
\end{equation*}
Thus one gets
\begin{equation}\label{eqn:HatDzu_def}
  \hat{D}(z,1+u) = \sum_{r \ge 0} g_{r}(z) u^{r} = \sum_{r \ge 0} u^{r} \cdot \frac{1}{r!} \sum_{n \ge 0} 2^{n} \E\big(\hat{X}_{n}^{\underline{r}}\big) z^{n},
\end{equation}
and in order to determine the asymptotic behaviour of the factorial (and raw) moments of $\hat{X}_{n}$ we carry out a local expansion of the functions $g_{r}(z) = [u^{r}] \hat{D}(z,1+u)$ around the dominant singularities followed by basic applications of so-called transfer lemmata.

The next lemma states the relevant properties of the coefficients of $\hat{D}(z,1+u)$.
\begin{lem}\label{lem:HatDzq_Expansion}
  Let $\hat{D}(z,q)$ be the g.f.\ of the shifted r.v.\ $\hat{X}_{n} = X_{n} - \frac{n}{2}$ as defined in \eqref{eqn:HatDzq_definition}. Then the functions $g_{r}(z) = [u^{r}] \hat{D}(z,1+u)$ obtained as coefficients in a series expansion of $\hat{D}(z,1+u)$ around $u=0$ have radius of convergence $\frac{1}{2}$ and, for $r \ge 1$, have the two dominant singularities $\rho_{1,2} = \pm \frac{1}{2}$. Moreover, the local behaviour of $g_{r}(z)$ around $\rho := \rho_{1} = \frac{1}{2}$ is given as follows, with $\mathcal{Z} := \frac{1}{1-2z}$:
\begin{equation*}
  g_{r}(z) = (r+1) \big(\frac{1}{8}\big)^{\frac{r}{2}} \mathcal{Z}^{\frac{r}{2} +1} \cdot \big(1+\mathcal{O}\big(\mathcal{Z}^{-\frac{1}{2}}\big)\big), \quad r \ge 0.
\end{equation*}
\end{lem}

\begin{remark}\label{rem:Second_Singularity}
  We remark that a closer inspection shows that the second dominant singularity $\rho_{2} = -\frac{1}{2}$ occurring in the functions $g_{r}(z)$ defined by Lemma~\ref{lem:HatDzq_Expansion} yield contributions that do not affect the main terms stemming from the contributions of the singularity $\rho = \rho_{1} = \frac{1}{2}$. Since we are here only interested in the main term contribution, we will restrict ourselves to elaborate the expansion around $\rho$. However, the presence of two dominant singularities is reflected by the fact, that lower order terms of the asymptotic expansions of the $r$-th moments of $X_{n}$ are different for $n$ even and $n$ odd, resp., as has been observed in \cite{KT2023}.
\end{remark}

\begin{proof}
Using \eqref{eqn:HatDzq_definition} and the explicit formula of $D(z,q)$ given in Lemma~\ref{lem:Dzq_formula}, one gets after simple manipulations
\begin{multline}\label{eqn:HatDzq_explicit}
  \hat{D}(z,q) = \frac{\sqrt{q} (\sqrt{q}-z)}{(\sqrt{q}-qz)^{2}} + \frac{2z(\sqrt{q}-z)}{(\sqrt{q}-(1+q)z)(\sqrt{q}-qz)}\\
		\mbox{} + \frac{z\big(\sqrt{q}(q-1) + \big((1+q)z-q^{\frac{3}{2}}\big)\big(1-2B(z^{2})\big)\big)}{(1-(1+q)B(z^{2}))(\sqrt{q}-(1+q)z)(\sqrt{q}-qz)}.
\end{multline}
We set $q=1+u$ and carry out a series expansion of the summands of \eqref{eqn:HatDzq_explicit} around $u=0$. Since this is a rather straightforward task using essentially the binomial series, but leads to rather lengthy computations when one intends to be exhaustive in every step, we will here only give a sketch of such computations and are omitting some of the details.

When treating the first summand in \eqref{eqn:HatDzq_explicit} and inspecting the coefficients in the series expansion around $u = q-1=0$, 
\begin{equation*}
  \hat{D}^{[1]}(z,q) := \frac{\sqrt{q} \, (\sqrt{q}-z)}{(\sqrt{q}-qz)^{2}} = \sum_{r \ge 0} g_{r}^{[1]}(z) u^{r},
\end{equation*}
one easily observes that the functions $g_{r}^{[1]}(z)$ are analytic for $|z|<1$ (to be more precise, the unique dominant singularity is at $z=1$), which causes exponentially small contributions for the coefficients $[z^{r}] g_{r}^{[1]}(z)$ compared to the remaining summands. Thus, these contributions are negligible and do not have to be considered further.

When expanding the second summand of \eqref{eqn:HatDzq_explicit} around $u=q-1=0$,
\begin{equation}\label{eqn:gr2_def}
  \hat{D}^{[2]}(z,q) := \frac{2z(\sqrt{q}-z)}{(\sqrt{q}-(1+q)z)(\sqrt{q}-qz)} = \sum_{r \ge 0} g_{r}^{[2]}(z) u^{r},
\end{equation}
we have to treat with more care the factor $(\sqrt{q}-(1+q)z)^{-1}$. First, by using the binomial series we get
\begin{equation*}
  \sqrt{q} - (1+q)z = \sqrt{1+u} - (2+u)z = (1-2z)\Big(1+\frac{u}{2} + \sum_{k \ge 2} \frac{c_{k}}{1-2z} u^{k}\Big),
\end{equation*}
with $c_{k} = \binom{\frac{1}{2}}{k}$, and further, by using the geometric series,
\begin{multline}\label{eqn:intermediate_expansion1}
  \frac{1}{\sqrt{q} - (1+q)z} = \frac{1}{(1-2z) \big(1+\frac{u}{2}\big(1+\sum_{k \ge 1} \frac{2 c_{k+1}}{1-2z} u^{k}\big)\big)}\\
	= \mathcal{Z} \Big(1+\sum_{\ell \ge 1} (-\frac{1}{2})^{\ell} u^{\ell} \big(1+\sum_{k \ge 1} 2 c_{k+1} \mathcal{Z} u^{k}\big)^{\ell}\Big).
\end{multline}
From this expansion it is apparent that all the coefficients of $u^{r}$ in the series expansion, considered as functions in $z$, have a unique dominant singularity at $z= \rho = \frac{1}{2}$. Furthermore, for $\ell \ge 1$ we obtain the following expansion in powers of $u$ and locally around $z= \rho$, i.e., $\mathcal{Z}=\infty$:
\begin{align*}
  \big(1+\sum_{k \ge 1} 2 c_{k+1} \mathcal{Z} u^{k}\big)^{\ell} & = 1 + \ell (2c_{2}) \mathcal{Z} u + \sum_{j=2}^{\ell} \binom{\ell}{j} (2c_{2})^{j} \mathcal{Z}^{j} \big(1+\mathcal{O}(\mathcal{Z}^{-1})\big)u^{j}\\
	& \quad \mbox{} + \ell(2c_{2})^{\ell-1} (2c_{3}) \mathcal{Z}^{\ell} \big(1+\mathcal{O}(\mathcal{Z}^{-1})\big)u^{\ell+1} + \sum_{k \ge \ell+2} \mathcal{O}(\mathcal{Z}^{\ell}) u^{k},
\end{align*}
which, after plugging into \eqref{eqn:intermediate_expansion1} and using $c_{2} = -\frac{1}{8}$, $c_{3}=\frac{1}{16}$ leads to the required expansion:
\begin{multline}\label{eqn:intermediate_expansion2}
  \frac{1}{\sqrt{q} - (1+q)z} = \mathcal{Z} - \frac{\mathcal{Z}}{2} u + \sum_{\ell \ge 1} \left[(\frac{1}{8})^{\ell} \mathcal{Z}^{\ell+1} (1+\mathcal{O}(\mathcal{Z}^{-1})) u^{2\ell}\right.\\
	\left.\mbox{} -\frac{1}{2} (\frac{1}{8})^{\ell} (2\ell+1) \mathcal{Z}^{\ell+1} (1+\mathcal{O}(\mathcal{Z}^{-1})) u^{2\ell+1}\right].
\end{multline}
Next, it is easy to see that the coefficients in the expansion around $u=0$ of the remaining factors of $\hat{D}^{[2]}(z,q)$ are functions in $z$ with radius of convergence $1$, and one gets
\begin{equation}\label{eqn:intermediate_expansion3}
  \frac{2z(\sqrt{q}-z)}{\sqrt{q}-qz} = 1+ \mathcal{O}(\mathcal{Z}^{-1}) + (1+ \mathcal{O}(\mathcal{Z}^{-1})) u + \sum_{r \ge 2} \mathcal{O}(\mathcal{Z}^{0}) u^{r}.
\end{equation}
Combining the expansions \eqref{eqn:intermediate_expansion2} and \eqref{eqn:intermediate_expansion3}, we obtain that the functions $g_{r}^{[2]}(z)$ in expansion~\eqref{eqn:gr2_def} have a unique dominant singularity at $z=\rho$ and allow there the local expansions
\begin{equation}\label{eqn:gr2_expansion}
  g_{r}^{[2]}(z) = \begin{cases}
	  (\frac{1}{8})^{\ell} \mathcal{Z}^{\ell+1} (1+\mathcal{O}(\mathcal{Z}^{-1})), & \quad \text{for $r = 2 \ell$ even},\\
		-\frac{1}{2} (\frac{1}{8})^{\ell} (2\ell-1) \mathcal{Z}^{\ell+1} (1+\mathcal{O}(\mathcal{Z}^{-1})), & \quad \text{for $r = 2 \ell+1$ odd}.
	\end{cases}
\end{equation}

Finally, we consider an expansion in powers of $u=q-1$ of the third summand of \eqref{eqn:HatDzq_explicit},
\begin{equation}\label{eqn:gr3_def}
  \hat{D}^{[3]}(z,q) := \frac{z\big(\sqrt{q}(q-1) + \big((1+q)z-q^{\frac{3}{2}}\big)\big(1-2B(z^{2})\big)\big)}{(1-(1+q)B(z^{2}))(\sqrt{q}-(1+q)z)(\sqrt{q}-qz)}.
\end{equation}
Let us define $\tilde{\mathcal{Z}} = \frac{1}{1-4z^{2}}$. Since $B(z^{2}) = \frac{1}{2}(1-\tilde{\mathcal{Z}}^{-\frac{1}{2}})$, we get
\begin{equation*}
  1-(1+q)B(z^{2}) = \tilde{\mathcal{Z}}^{-\frac{1}{2}} \Big(1-\frac{1}{2}(\tilde{\mathcal{Z}}^{\frac{1}{2}}-1)\Big)
\end{equation*}
and thus
\begin{equation}\label{eqn:gr3z_intermediate}
  \frac{1}{1-(1+q)B(z^{2})} = \frac{\tilde{\mathcal{Z}}^{\frac{1}{2}}}{1-\frac{1}{2}(\tilde{\mathcal{Z}}^{\frac{1}{2}}-1)u}
	= \tilde{\mathcal{Z}}^{\frac{1}{2}} \Big(1+\sum_{r \ge 1}\Big(\frac{1}{2}(\tilde{\mathcal{Z}}^{\frac{1}{2}}-1)u\Big)^{r}\Big).
\end{equation}
Therefore, for this factor of $\hat{D}^{[3]}(z,q)$ we obtain that the coefficients of $u^{r}$ are functions in $z$ with two dominant singularities $\rho_{1,2} = \pm \frac{1}{2}$. However, as already pointed out in Remark~\ref{rem:Second_Singularity}, the contributions stemming from the singularity $\rho_{2}=-\frac{1}{2}$ do not affect the main term contributions and thus they are not considered any further. Since $\tilde{\mathcal{Z}} = \frac{1}{(1-2z)(1+2z)} = \frac{1}{2} \mathcal{Z} (1+\mathcal{O}(\mathcal{Z}^{-1}))$, we thus obtain from \eqref{eqn:gr3z_intermediate} the local expansion around $z=\rho$:
\begin{equation}\label{eqn:intermediate_expansion4}
  \frac{1}{1-(1+q)B(z^{2})} = \sum_{r \ge 0} (\frac{1}{2})^{\frac{3r+1}{2}} \mathcal{Z}^{\frac{r+1}{2}} (1+\mathcal{O}(\mathcal{Z}^{-\frac{1}{2}})) u^{r}.
\end{equation}
In a similar fashion one obtains the expansion
\begin{multline}\label{eqn:intermediate_expansion5}
  \frac{z\big(\sqrt{q}(q-1) + \big((1+q)z-q^{\frac{3}{2}}\big)\big(1-2B(z^{2})\big)\big)}{\sqrt{q}-qz}\\
	= -2^{\frac{1}{2}} \mathcal{Z}^{-\frac{3}{2}} (1+\mathcal{O}(\mathcal{Z}^{-1})) + (1+\mathcal{O}(\mathcal{Z}^{-\frac{1}{2}})) u + \sum_{r \ge 2} \mathcal{O}(\mathcal{Z}^{0}) u^{r},
\end{multline}
whereas the last factor of $\hat{D}^{[3]}(z,q)$ has been treated already in \eqref{eqn:intermediate_expansion2}. Combining expansions \eqref{eqn:intermediate_expansion4}, \eqref{eqn:intermediate_expansion5} and \eqref{eqn:intermediate_expansion2}, we get
\begin{align}\label{eqn:HatD3zq_expansion1}
  & \hat{D}^{[3]}(z,q) = \Big(\sum_{r \ge 0} (\frac{1}{8})^{\frac{r}{2}} \mathcal{Z}^{\frac{r}{2}} (1+\mathcal{O}(\mathcal{Z}^{-\frac{1}{2}})) u^{r}\Big)\\
	& \qquad \cdot \Big(\sum_{\ell \ge 0} (\frac{1}{8})^{\ell} \mathcal{Z}^{\ell} (1+\mathcal{O}(\mathcal{Z}^{-1})) u^{2\ell} + (-\frac{1}{2}) (\frac{1}{8})^{\ell} (2\ell+1) \mathcal{Z}^{\ell} (1+\mathcal{O}(\mathcal{Z}^{-1})) u^{2\ell+1}\Big)\notag\\
	& \qquad \cdot \Big(-(1+\mathcal{O}(\mathcal{Z}^{-1})) +(\frac{1}{2})^{\frac{1}{2}} \mathcal{Z}^{\frac{3}{2}} (1+\mathcal{O}(\mathcal{Z}^{-\frac{1}{2}})) u + \sum_{r \ge 2} \mathcal{O}(\mathcal{Z}^{\frac{3}{2}}) u^{r}\Big).\notag
\end{align}
To compute the Cauchy product of the first two factors of \eqref{eqn:HatD3zq_expansion1} we use (with some coefficients $\alpha_{r}, \beta_{r} \in \mathbb{R}$):
\begin{multline*}
  \Big(\sum_{r \ge 0} \alpha_{r} \mathcal{Z}^{\frac{r}{2}} (1+\mathcal{O}(\mathcal{Z}^{-\frac{1}{2}})) u^{r} \Big)
	\cdot \Big(\sum_{r \ge 0} \beta_{r} \mathcal{Z}^{\lfloor \frac{r}{2} \rfloor} (1+\mathcal{O}(\mathcal{Z}^{-\frac{1}{2}})) u^{r} \Big)\\
	= \sum_{r \ge 0} \gamma_{r} \mathcal{Z}^{\frac{r}{2}} (1+\mathcal{O}(\mathcal{Z}^{-\frac{1}{2}})) u^{r}, \qquad \text{with} \quad \gamma_{r} = \sum_{\ell=0}^{\lfloor \frac{r}{2} \rfloor} \beta_{2\ell} \, \alpha_{r-2\ell}.
\end{multline*}
In particular, for $\alpha_{r} = (1/8)^{\frac{r}{2}}$ and $\beta_{2\ell} = (1/8)^{\ell}$ one gets $\gamma_{r} = (1/8)^{\frac{r}{2}} (\lfloor r/2 \rfloor+1)$, which eventually shows that the coefficients $g_{r}^{[3]}(z)$ in the expansion of $\hat{D}^{[3]}(z,q)$ around $u=q-1=0$ are given as follows:
\begin{equation}\label{eqn:gr3_expansion}
  g_{r}^{[3]}(z) = \begin{cases}
	  -(1+\mathcal{O}(\mathcal{Z}^{-1})), & \quad \text{for $r=0$},\\
		2 (\frac{1}{8})^{\frac{r}{2}} (\lfloor \frac{r-1}{2} \rfloor + 1) \mathcal{Z}^{\frac{r}{2}+1} (1+\mathcal{O}(\mathcal{Z}^{-\frac{1}{2}})), & \quad \text{for $r \ge 1$}.
	\end{cases}
\end{equation}

Thus, combining \eqref{eqn:gr2_expansion} and \eqref{eqn:gr3_expansion} one obtains, after simple manipulations, the stated local expansion of the coefficients $g_{r}(z)=g_{r}^{[1]}(z)+g_{r}^{[2]}(z)+g_{r}^{[3]}(z)$ in the series expansion of $\hat{D}(z,q)$ around $u=q-1=0$.
\end{proof}
The expansion of $\hat{D}(z,q)$ stated in Lemma~\ref{lem:HatDzq_Expansion} easily yields the asymptotic behaviour of the moments of $Y_{n}$.
\begin{proof}[Proof of Theorem~\ref{the:Yn_Moments}]
According to the definition of $\hat{X}_{n}$ and relation~\eqref{eqn:HatDzu_def} we get for the factorial moments:
\begin{equation*}
  \E\big(\hat{X}_{n}^{\underline{r}}\big) = \frac{r! [z^{n} u^{r}] \hat{D}(z,1+u)}{2^{n}} = \frac{r! [z^{n}] g_{r}(z)}{2^{n}},
\end{equation*}
with $g_{r}(z)$ as defined in Lemma~\ref{lem:HatDzq_Expansion}. Since the dominant singularity of $g_{r}(z)$ relevant for the asymptotic behaviour of the main term is at $z=\rho=\frac{1}{2}$ (see Remark~\ref{rem:Second_Singularity}) with a local expansion stated in above lemma, we can apply basic transfer lemmata~\cite{FlaSed} to obtain for the coefficients:
\begin{align*}
  [z^{n}] g_{r}(z) & = [z^{n}] (r+1) (\frac{1}{8})^{\frac{r}{2}} \frac{1}{(1-2z)^{\frac{r}{2}+1}} \cdot \big(1+\mathcal{O}(\sqrt{1-2z})\big)\\
	& = (r+1) (\frac{1}{8})^{\frac{r}{2}} \frac{2^{n} n^{\frac{r}{2}}}{\Gamma(\frac{r}{2}+1)} \cdot \big(1+\mathcal{O}(n^{-\frac{1}{2}})\big).
\end{align*}
Thus, the asymptotic behaviour of the factorial moments is given by
\begin{equation}\label{eqn:HatXnr_factorial_asymptotics}
  \E\big(\hat{X}_{n}^{\underline{r}}\big) = \frac{(r+1)! (\frac{1}{8})^{\frac{r}{2}}}{\Gamma(\frac{r}{2}+1)} n^{\frac{r}{2}} \cdot \big(1+\mathcal{O}(n^{-\frac{1}{2}})\big), \quad r \ge 0.
\end{equation}
Since the $r$-th integer moments can be obtained by a linear combination of the factorial moments of order $\le r$, due to
$\E\big(\hat{X}_{n}^{r}\big) = \E\big(\hat{X}_{n}^{\underline{r}}\big) + \mathcal{O}\big(\E\big(\hat{X}_{n}^{\underline{r-1}}\big)\big) = \E\big(\hat{X}_{n}^{\underline{r}}\big) \cdot \big(1+\mathcal{O}(n^{-\frac{1}{2}})\big)$ the same asymptotic behaviour \eqref{eqn:HatXnr_factorial_asymptotics} also holds for the raw moments. An application of the duplication formula for the $\Gamma$-function gives then the alternative representation
\begin{equation}\label{eqn:HatXnr_raw_asymptotics}
  \E\big(\hat{X}_{n}^{r}\big) = \frac{\Gamma\big(\frac{r+3}{2}\big)}{2^{\frac{r}{2}-1} \sqrt{\pi}} \, n^{\frac{r}{2}} \cdot \big(1+\mathcal{O}(n^{-\frac{1}{2}})\big), \quad r \ge 0.
\end{equation}
Since $\hat{X}_{n} = X_{n} - n/2 = \sqrt{n} \, Y_{n}$, equation~\eqref{eqn:HatXnr_raw_asymptotics} implies $\E(Y_{n}^{r}) \to \E(G^{r})$ as stated.
\end{proof}

\section{First pure luck guess}
So far, we have been interested in the total number of correct guesses. 
As the guesser follows the optimal strategy, the chances of a correct guess
are always greater or equal $50$ percent. Starting with a deck of $n$ cards, we might be interested in the number of cards $P_n$ (divided by two) remaining in the deck when the first ``pure luck guess'' with only a $50$ percent success chance occurs. By Proposition~\ref{Prop:firstCard} and Theorem~\ref{the:distDecomp}, this can only happen after the ``first phase'' of always guessing the smallest number remaining in the deck has failed and thus finished and so the ``two-color card guessing process'' has been started already. 
Similar to Theorem~\ref{the:distDecomp} we obtain for $P := P_{n}$ the distributional equation
\[
P_n \law I_1 \cdot P^{\ast}_{n-1}
+(1-I_1)(1-I_2)\cdot H_{n-1-J_n,J_n},
\]
where $I_1\law \Be(0.5)$, $I_2 \law \Be(0.5^{n-1})$, and $H_{m_1,m_2}$ denotes the number of cards present, divided by two, in a two-color card guessing game when for the first time a pure luck guess occurs. Additionally, $P^{\ast}_{n-1}$ is an independent copy of $P$ defined on $n-1$ cards. Moreover, as in Theorem~\ref{the:distDecomp}, $J_n\law \Bin^{*}(n-1,p)$ denotes a truncated binomial distribution: 
\[
\P(J_n=j)=\binom{n-1}{j}/(2^{n-1}-1),\quad  0\le j \le n-2.
\]
All random variables $I_1$, $I_2$, $J_n$, as well as $H_{m_1,m_2}$ are mutually independent.

\smallskip

We use a limit law for $H_{m_1,m_2}$, for a certain regime of $m_1,m_2$ when both parameters tending to infinity, relying on results of~\cite{PK2023,KuPanPro2009}. 

First, we require a new distribution, a functional of a L\'evy distributed random variable $L=\Levy(c)$, $c>0$, with density
\begin{equation}
\label{eqn:levy}
f_L(x)=\sqrt{\frac{c}{2\pi}}\frac{e^{-c/(2x)}}{x^{3/2}},\quad x>0.
\end{equation}
\begin{defi}[Reciprocal of a shifted L\'evy distribution]
Let $L=\Levy(c)$, $c>0$. Then, let $R=R(c)$ denote the reciprocal
of the shifted random variable $1+L$:
\[
R= \frac{1}{1+L},\quad \text{with support $(0,1)$}.
\]
The density of $R$ is given by 
\[
f_R(x)=\sqrt{\frac{c}{2\pi}}\cdot\frac1{(1-x)^{3/2}x^{1/2}}\cdot e^{-\frac{cx}{2(1-x)}},\quad 0<x<1.
\]
\end{defi}
This random variable in terms of above density function has been appeared already in several applications.
See for example~\cite{KuPanPro2009} for the limit law of the hitting time in sampling without replacement
or~\cite{HPW2022,HPW2023} for its occurrence in the limit law of an uncover process for random trees.
Moreover, this random variable has appeared earlier in context of the standard additive coalescent,
where also the relation to the L\'evy distribution has been observed by Aldous and Pitman~\cite[Corollary 5 and Theorem 6]{AldousPitman1998}.
We also note the random variable appears as the limit law 
of random dynamics on the edges of a uniform Cayley tree, a so-called "fire on tree" model~\cite{Bertoin2012}. 
In contrast to the L\'evy distribution, the random variable $R$ has integer moments of all orders. In the special case of $c=1$ the moments have a particularly interesting structure~\cite[Lemma 3]{Bertoin2012}:
\[
\E(R^k)=\E\big(\exp(-\chi(2k))\big),
\]
where $\chi(2k)$ is a chi-variable with $2k$ degrees of freedom, with density
\[
\frac{2^{1-k}}{(k-1)!}x^{2k-1}\exp(-x^2/2)dx,\quad x\ge 0.
\]

Finally, we note that it is easy to see that $R$ has the stated density function:
\begin{align*}
F_R(x)&=\P\{R\le x\}=\P\Big\{\frac1{1+L}\le x\Big\}=\P\Big\{\frac1x\le 1+L\Big\}\\
&=\P\Big\{L\ge \frac{1}x-1\Big\}=1-\P\Big\{L < \frac{1-x}{x}\Big\}.
\end{align*}
Consequently, 
\begin{align*}
f_R(x)&=-f_L\big((1-x)/x\big)\cdot (-1)\cdot x^{-2}=\sqrt{\frac{c}{2\pi}}\frac{e^{-cx/(2(1-x))}x^{3/2}}{(1-x)^{3/2}}\cdot \frac1{x^{2}},
\end{align*}
immediately leading to the stated density.

\smallskip

Next, we use the following result. 
\begin{lem}[Hitting time and first pure luck guess]
\label{eqn:lemHit}
Let $H_{m_1,m_2}$ denote the random variable counting the number of remaining cards, divided by two, when for the first time
a pure luck guess happens in the two-color card guessing game, starting with $m_1$ red and $m_2$ black cards. 
Assume further that $m_1,m_2\to\infty$ and $m_2=m_1-\rho\sqrt{m_1}$, with $\rho > 0$.
Then, 
\[
\frac{H_{m_1,m_2}}{m_1}\claw R(\rho^2/2).
\]
\end{lem}

\begin{proof}
We combine arguments of~\cite{PK2023,KuPanPro2009}: by the results of~\cite{PK2023}, the weighted sample paths of the two-color card guessing game coincide with the sample paths of the sampling without replacement urn (see also Remark~\ref{rem:Urn_LatticPath}). In particular, this holds with respect
to the hitting position of the diagonal $x=y$, as a crossing of the diagonal without hitting cannot happen. 
In~\cite{KuPanPro2009} such hitting positions have been studied in a general setting for paths starting at $(m_{1},m_{2})$, with $m_{1} \ge t m_{2} + s$, and absorbing lines $y=x/t -s/t$, for $t\in\N$ and $s\in\N_0$. For our purpose we set $t=1$ and $s=0$ in~\cite[Theorem 2 (4)]{KuPanPro2009}, which gives for $0 < x < 1$:
\begin{equation*}
  \P\Big\{\frac{H_{m_{1},m_{2}}}{m_{1}} \le x\Big\} \sim \int_{0}^{x} \frac{\rho}{\sqrt{2} \, \sqrt{2\pi}} \frac{1}{\sqrt{u} \, (1-u)^{\frac{3}{2}}} \cdot e^{-\frac{\rho^{2} u}{4(1-u)}} du = \int_{0}^{x} f_{R}(u) du,
\end{equation*}
with $f_{R}(x)$ the density of the reciprocal of a shifted L\'evy distribution with parameter $c=\rho^{2}/2$. Thus, this shows the stated limit law.
\end{proof}
In order to obtain the limit law of $P_n$ we require the limit law of $H_{n-1-J_n,J_n}$, which will be determined next.
\begin{lem}
\label{lem:twoColorBinBetaLimit}
The random variable $H_{n-1-J_n,J_n}$ has an Arcsine limit law $\beta(\frac12,\frac12)$:
\[
\frac{H_{n-1-J_n,J_n}}{\frac{n}2} \claw \beta\big(\frac12,\frac12\big),
\]
i.e., after suitable scaling, it converges in distribution to a Beta-distributed r.v.\ with parameters $1/2$ and $1/2$ that has the probability density function
\begin{equation*}
  f_{\beta}(x) = \frac{1}{\pi} \frac{1}{\sqrt{x (1-x)}}, \quad 0 < x < 1.
\end{equation*}
\end{lem}

In a way analogous to the proof of Theorem~\ref{the:Yn_LimitLaw}, this lemma readily leads to the main result of this section. 
\begin{theorem}
The random variable $P_n$ counting the number of remaining cards, divided by two, when the first pure luck guess with only a $50$ percent success chance occurs, starting with $n$ ordered cards and performing a single riffle shuffle, has a $\beta\big(\frac12,\frac12\big)$ limit law, a so-called Arcsine distribution:
\[
\frac{P_n}{\frac{n}2}\to \beta\big(\frac12,\frac12\big).
\]
\end{theorem}

\begin{proof}[Proof of Lemma~\ref{lem:twoColorBinBetaLimit}]
We proceed similar to the proof of Lemma~\ref{lem:twoColorBinLimit}. 
We study the distribution function $F(k)=\P\{H_{n-1-J_n,J_n}\le k\}$ and obtain
\[
F(k)=\sum_{j=0}^{n-2}\frac{\binom{n-1}j}{2^{n-1}-1}\P\{H_{n-1-j,j}\le k\}.
\]
We use the symmetry of the binomial distribution around $\lfloor n/2 \rfloor$ as well as $H_{m_{1},m_{2}} = H_{m_{2},m_{1}}$ and approximate the binomial distribution using the de~Moivre-Laplace theorem. 
This leads to 
\[
F(k)\sim 2 \int_{\lfloor n/2 \rfloor}^n e^{-\frac{(j-\mu_n)^2}{2\sigma_n^2}}\cdot \frac{1}{\sigma_n\sqrt{2\pi}} \cdot
\P\{H_{j,n-1-j}\le k\}dj,
\]
where $\mu_n=n/2$ and $\sigma_n=\sqrt{n}/2$. Changing the range of integration and the choice $k=x\cdot n/2$, with $0<x<1$,
leads then, together with Lemma~\ref{eqn:lemHit}, to the improper integral
\begin{align*}
F(k)&=\P\Big\{\frac{H_{n-1-J_n,J_n}}{n/2}\le x\Big\}\\
&\sim 2\int_0^{\infty} e^{-t^2/2}\frac{1}{\sqrt{2\pi}}\int_0^{x}\frac{t}{\sqrt{2\pi u}(1-u)^{3/2}}\cdot e^{-\frac{t^2 u}{2(1-u)}}du\ dt.
\end{align*}
Derivation with respect to $x$ gives then the desired density function, 
where the arising improper integral is readily evaluated:
\[
\int_0^{\infty}e^{-t^2/2}\cdot t\cdot e^{-t^{2}g/2}dt=\frac{1}{1+g}, \qquad \text{for $g \ge 0$}.
\]
Setting $g=x/(1-x)$ immediately yields the Arcsine law density function
\[
  f_{\beta}(x) = \frac1{\pi}\frac1{\sqrt{x(1-x)}},\quad 0<x<1. 
\]
\end{proof}

Finally, we note that the number of guesses with success probability one, i.e., where the guesser knows in advance to be correct, 
can be treated in a similar way.

\section*{Declarations of interest}
The authors declare that they have no competing financial or personal interests that influenced the work reported in this paper. 

\bibliographystyle{cyrbiburl}
\bibliography{CardGuessingOneRiffle-refs}{}

\begin{thebibliography}{10}

\bibitem{AldousPitman1998}
D.~J. Aldous and J.~Pitman.
\newblock The standard additive coalescent.
\newblock  \textit{Ann. Probab.}, 26:1703--1726, 1998.

\bibitem{Bertoin2012}
J.~Bertoin.
\newblock Fires on trees.
\newblock  \textit{Annales de l'I.H.P. Probabilités et statistiques},
  48(4):909--921, 2021.

\bibitem{BlackwellHodges1957}
D.~Blackwell and J.~L.~Hodges Jr.
\newblock Design for the control of selection bias.
\newblock  \textit{The Annals of Mathematical Statistics}, 28(2):449--460,
  1957.

\bibitem{Diaconis1978}
P.~Diaconis.
\newblock Statistical problems in esp research.
\newblock  \textit{Science}, 201(4351):131--136, 1978.

\bibitem{DiaconisGraham1981}
P.~Diaconis and R.~Graham.
\newblock The analysis of sequential experiments with feedback to subjects.
\newblock  \textit{Annals of Statistics}, 9(1):3--23, 1981.

\bibitem{Efron1971}
B.~Efron.
\newblock Forcing a sequential experiment to be balanced.
\newblock  \textit{Biometrika}, 58(3):403--417, 1971.

\bibitem{Fisher1936}
Ronald~A. Fisher.
\newblock Design of experiments.
\newblock  \textit{British Medical Journal}, 1(3923):554, 1936.

\bibitem{FlaSed}
P.~Flajolet and R.~Sedgewick.
\newblock  \textit{Analytic Combinatorics}.
\newblock Cambridge University Press, 2009.

\bibitem{Gilbert1955}
E.~Gilbert.
\newblock Theory of shuffling.
\newblock Technical memorandum, Bell Labs, 1955.

\bibitem{HPW2022}
B.~Hackl, A.~Panholzer, and S.~Wagner.
\newblock Uncovering a random tree.
\newblock In Mark~Daniel Ward, editor,  \textit{33rd International Conference
  on Probabilistic, Combinatorial and Asymptotic Methods for the Analysis of
  Algorithms (AofA 2022), Dagstuhl, Germany}, volume 225 of  \textit{Leibniz
  International Proceedings in Informatics (LIPIcs)}, page 10:1–10:17,
  Dagstuhl, Germany, 2022. Schloss Dagstuhl -- Leibniz-Zentrum f{\"u}r
  Informatik.

\bibitem{HPW2023}
B.~Hackl, A.~Panholzer, and S.~Wagner.
\newblock \href{https://arxiv.org/abs/2301.00664}{The uncover process for
  random labeled trees}.
\newblock Manuscript (Arxiv), 2023.

\bibitem{HeOttolini2021}
J.~He and A.~Ottolini.
\newblock \href{https://arxiv.org/abs/2108.07355}{Card guessing and the
  birthday problem for sampling without replacement}.
\newblock Manuscript (Arxiv), 2021.

\bibitem{KnoPro2001}
A.~Knopfmacher and H.~Prodinger.
\newblock A simple card guessing game revisited.
\newblock  \textit{Electronic Journal of Combinatorics}, 8, R13:9 pages, 2001.

\bibitem{KT2023}
T.~Krityakierne and T.~A. Thanatipanonda.
\newblock \href{https://arxiv.org/abs/2107.11142}{The card guessing game: A
  generating function approach}.
\newblock  \textit{Journal of Symbolic Computation}, 115:1--17, 2023.

\bibitem{PK2023}
M.~Kuba and A.~Panholzer.
\newblock \href{https://arxiv.org/abs/2303.04609}{On card guessing with two
  types of cards}.
\newblock Manuscript (Arxiv), 2023.

\bibitem{KuPanPro2009}
M.~Kuba, A.~Panholzer, and H.~Prodinger.
\newblock Lattice paths, sampling without replacement, and limiting
  distributions.
\newblock  \textit{Electronic Journal of Combinatorics}, 16 (1), R67:12 pages,
  2009.

\bibitem{Leva1988}
K.~Levasseur.
\newblock How to beat your kids at their own game.
\newblock  \textit{Mathematical Magazine}, 61:301--305, 1988.

\bibitem{Liu2021}
P.~Liu.
\newblock On card guessing game with one time riffle shuffle and complete
  feedback.
\newblock  \textit{Discrete Applied Mathematics}, 288:270--278, 2021.

\bibitem{Loe1977}
M.~Lo{\`{e}}ve.
\newblock  \textit{Probability Theory I}.
\newblock Springer, 4th edition, 1977.

\bibitem{OttoliniSteiner2022}
A.~Ottolini and S.~Steinerberger.
\newblock \href{https://arxiv.org/abs/2211.09094}{Guessing cards with complete
  feedback}.
\newblock Manuscript (Arxiv), 2022.

\bibitem{OT2023}
A.~Ottolini and R.~Tripathi.
\newblock \href{https://arxiv.org/abs/2303.15601}{Central limit theorem in
  complete feedback games}.
\newblock Manuscript (Arxiv), 2023.

\bibitem{Read1962}
R.~C. Read.
\newblock Card-guessing with information. a problem in probability.
\newblock  \textit{American Mathematical Monthly}, 69:506--511, 1962.

\bibitem{Zagier1990}
D.~Zagier.
\newblock How often should you beat your kids?
\newblock  \textit{Mathematical Magazine}, 63:89–92, 1990.

\end{thebibliography}


\end{document}